\documentclass[10pt,reqno]{amsart}
\usepackage{amsfonts}
\usepackage{amssymb}
\usepackage{graphicx}
\usepackage{newlfont}
\usepackage{mathrsfs}
\usepackage{graphics}
\usepackage{amsmath, amssymb, amsfonts,amsthm}
\usepackage{textcomp}
\usepackage{color}
\usepackage{mathtools}
\usepackage{multicol}
\usepackage{bbm}
\usepackage{float}
\usepackage{enumerate}

\usepackage[numbers,sort&compress]{natbib}

\newtheorem{theorem}{Theorem}[section]
\newtheorem{cor}[theorem]{Corollary}

\theoremstyle{remark}

\newtheorem{definition}[theorem]{Definition}

\newcommand{\cd}{\ \Rightarrow \ }
\newcommand{\cp}{\ \stackrel{P}{\longrightarrow} \ }

\newcommand{\ed}{\ \stackrel{d}{=} \ }

\newcommand{\TT}{\mbox{${\mathcal T}$}}
\newcommand{\PP}{\mbox{${\mathcal P}$}}

\newcommand{\GG}{\mbox{${\mathcal G}$}}
\newcommand{\FF}{\mbox{${\mathcal F}$}}

\newcommand{\II}{\mbox{${\mathcal I}$}}
\newcommand{\BB}{\mbox{${\mathcal B}$}}

\newcommand{\MM}{\mbox{${\mathcal M}$}}
\renewcommand{\SS}{\mbox{${\mathcal S}$}}

\newcommand{\Rbold}{\mbox{${\mathbb R}$}}

\newcommand{\Zbold}{\mbox{${\mathbb Z}$}}

\newcommand{\Pbold}{\mbox{${\mathbb P}$}}

\newcommand{\Hbold}{\mbox{${\mathbb H}$}}

\newcommand{\bE}{{\mathbf E}}
\newcommand{\bP}{{\mathbf P}}

\newcommand{\bv}{{\mathbf v}}
\newcommand{\bp}{{\mathbf p}}

\newcommand{\Var}{\mathbf{var}}
\newcommand{\eps}{\epsilon}

\newcommand{\bone}{\mathbf{1}}

\newcommand{\el}{\stackrel{d}{=}}

\newcommand{\bR}{\mathbb{R}}
\newcommand{\bmu}{\mathbf{\mu}}
\newcommand{\bN}{\mathbb{N}}
\newcommand{\bZ}{\mathbb{Z}}
\newcommand{\bzero}{\mathbf{0}}

\begin{document}

\title[A New Approach to P\'olya Urn Schemes]{A New Approach to P\'olya Urn Schemes and Its Infinite Color Generalization}
\author{Antar Bandyopadhyay} 
\address[Antar Bandyopadhyay]{Theoretical Statistics and Mathematics Unit \\
         Indian Statistical Institute, Delhi Centre \\ 
         7 S. J. S. Sansanwal Marg \\
         New Delhi 110016 \\
         INDIA}
\address{Theoretical Statistics and Mathematics Unit, 
         Indian Statistical Institute, Kolkata;
         203 B. T. Road, Kolkata 700108, INDIA}
\email{antar@isid.ac.in}          
\author{Debleena Thacker}  
\address[Debleena Thacker]{Mathematical Statistics,\\
         Centre for Mathematical Sciences, Lund University.
          SE-22100, 
          Lund, Sweden\\}
         
\email{thackerdebleena@gmail.com}

\begin{abstract}
In this work we generalize P\'{o}lya urn schemes with possibly infinitely many colors and extend the earlier models 
described in \cite{BlackMac73, BaTh2013, BaTh14}. 
We provide a novel and unique approach of representing the observed sequence of colors in terms 
a \emph{branching Markov chain} on \emph{random recursion tree}. 
This enables us to derive fairly general asymptotic for our urn schemes. 
We then illustrate through several examples that our method can easily derive the classical results for finite urns, 
as well as, many new results for infinite color urns. 
\end{abstract}


                                             
\maketitle

\section{Introduction}
\label{Sec: Intro}
 
In recent days various \emph{urn schemes} and their many generalizations have been a key element of study for 
random processes with reinforcements
\cite{Svante2, BaiHu05, Pe07, FlDuPu06, maulik1, DasMau11, CrGeVolWaWa11, LauLimi12, 
CoCoLi13, LaPa13, ChHsYa14, BaTh2013, BaTh14}. 
Starting from the seminal work by P\'olya \cite{Polya30}, 
various types of urn schemes with finitely many colors have been 
widely studied in literature
\cite{Fri49, Free65, AthKar68, BagPal85, Pe07, Pe90, Gouet, Svante1, Svante2,
BaiHu05, FlDuPu06, maulik1, maulik2, DasMau11, ChKu13, ChHsYa14}.
\cite{Pe07} provides an extensive survey of the known results.
However, other than the
classical work by Blackwell and MacQueen \cite{BlackMac73}, there has not been much development of infinite
color generalization of the P\'olya urn scheme. Recently the authors studied a specific class of
urn models with infinitely many colors where the color set is indexed by the $d$-dimensional integer lattice $\Zbold^d$, 
\cite{BaTh2013, BaTh14}. These works nicely complement the work \cite{BlackMac73} by introducing examples of 
infinite color schemes with ``\emph{off-diagonal}'' entries and showed that the asymptotic behavior is essentially 
determined by an underlying random walk. 

In this paper, we further generalize urn schemes with colors index by an arbitrary set $S$ endowed with a $\sigma$-algebra 
$\SS$. As we will see in the sequel, the classical models can be realized as a sub-model when $S$ is finite and in that case
$\SS$ will simply be the power set of $S$, which we will denote by $\wp\left(S\right)$. 
The non-classical case discussed in 
\cite{BlackMac73} can also be obtained by appropriately choosing the measurable space $\left(S, \SS\right)$ as the 
\emph{Borel space} of a Polish space $S$. Further the models described in \cite{BaTh2013, BaTh14} can be obtained by choosing
$S = \Zbold^d$ and $\SS = \wp\left(\Zbold^d\right)$. 

However, it is worthwhile to note that we will only consider \emph{balanced} urn schemes. 
For $S$ countable (finite or infinite), it means that if 
$R := \left(\left(R\left(i,j\right)\right)\right)_{i,j \in S}$ denotes the \emph{replacement matrix}, that is, 
$R\left(i,j\right) \geq 0$ is the \emph{number of balls of color $j$ to be placed in the the urn when the color of the selected ball 
is $i$}, then
for a balanced urn, all row sums of $R$ are constant. In this case, without loss of any generality, 
it is somewhat customary to assume that that $R$ is a \emph{stochastic matrix} \cite{maulik1, maulik2}. For more general $S$ we refer to the
next subsection for further details. 

The main contribution of this work is two ``\emph{representations}'' of the urn model with a Markov chain on $S$, 
which we will call the \emph{associated Markov chain}. 
These representations are novel and useful in deriving asymptotic results for the expected and random configurations of the urn.

There are few standard methods for analyzing finite color urn models which are mainly based on 
martingale techniques \cite{Gouet, maulik1, maulik2, DasMau11}, stochastic approximations \cite{LaPa2013} 
and embedding into continuous time pure birth processes \cite{AthKar68, Svante1, Svante2, BaiHu05}. 
Typically the analysis of a finite color urn is heavily dependent on the 
\emph{Perron-Frobenius theory} \cite{Sene06} of matrices with positive entries
and \emph{Jordon Decomposition} \cite{Curtis93} of finite dimensional matrices
\cite{AthKar68, Gouet, Svante1, Svante2, BaiHu05, maulik1, DasMau11}. 
The absence of such a theory when $S$ is infinite, makes the analysis of urn with infinitely many 
colors quite difficult and challenging. Instead in \cite{BlackMac73} the results were derived using \emph{exchangeability}
of the observed sequence of colors. However, as observed in \cite{BaTh2013}, exchangeability fails in the 
presence of off-diagonal entries and in \cite{BaTh2013}, the authors took a different approach of 
marginally embedding the observed sequence of colors to the underlying random walk sequence. The major
contribution of this work is to further this embedding for any general urn scheme and derive 
asymptotic results by bypassing the standard martingale and matrix theoretic techniques.

\subsection{Model}
\label{Subsec: Model}
We consider the following generalization of P\'olya urn scheme where
the colors are indexed by a non-empty subset $S$ of $\Rbold^d$ for some $d \geq 1$, such that, 
under subspace topology $S$ is a Polish space. A necessary and sufficient for $S$ to be Polish is that it is a $G_{\delta}-$set, that is, $S$ is a countable intersection of open sets, \cite{Boorbaki71}.
We endow $S$ with the corresponding Borel $\sigma$-algebra and denote it by $\SS$. 
Let $\MM\left(S\right)$ and $\PP\left(S\right)$ denote respectively the set of all \emph{finite measures} and 
the set of all \emph{probability measures} on the measurable space $\left(S, \SS\right)$. 
Note that the classical case when $S$ is finite or the non-classical cases discussed in \cite{BaTh2013, BaTh14}
are obtained by taking $S$ a discrete subset of $\Rbold^d$ of appropriate cardinality. 

Let $R : S \times \SS \rightarrow \left[0,1\right]$ be a Markov kernel on $S$, that is, 
for every $s \in S$, as a set-function of $\SS$, $R\left(s, \cdot\right)$ is a probability measure on $\left(S, \SS\right)$; 
and for every $A \in \SS$, the function $s \mapsto R\left(s, A\right)$ is $\SS / \BB_{\left[0,1\right]}$-measurable. 

By a \emph{configuration} of the urn at time $n \geq 0$, we will consider a finite measure
$U_n \in \MM\left(S\right)$, such that, if $Z_n$ represents the randomly chosen color at the $\left(n+1\right)$-th draw
then the conditional distribution of $Z_n$ given the ``\emph{past}'', is given by
\[
\Pbold\left( Z_n \in ds \,\Big\vert\, U_n, U_{n-1}, \cdots, U_0 \right) \propto U_n\left(ds\right).
\]

Formally, starting with $U_0 \in \PP\left(S\right)$ we define 
$\left(U_n\right)_{n \geq 0} \subseteq \MM\left(S\right)$ recursively as follows
\begin{equation}
\label{Equ:Fundamental-Recurssion}
U_{n+1} \left(A\right) = U_n\left(A\right) + R\left(Z_n, A\right), \quad\quad A \in \SS,
\end{equation}
and
\begin{equation}
\Pbold\left( Z_n \in ds \,\Big\vert\, U_n, U_{n-1}, \cdots, U_0 \right) = \frac{U_n\left(ds\right)}{n+1}.
\label{Equ:Conditional-Distribution-of-Zn}
\end{equation}

Notice that, if $S$ is countable then $R$ can be presented as a \emph{stochastic matrix} and then
$R\left(Z_n, \cdot\right)$ is the  $Z_n$-th row of the replacement matrix $R$. 
We will refer to 
the process $\left(U_n\right)_{n \geq 0}$ as the \emph{urn model} with colors index by $S$, 
initial configuration $U_0$ and replacement kernel $R$.

It is worth mentioning here that a little more general model may be obtained by taking $U_0 \in \MM\left(S\right)$ and not just
in $\PP\left(S\right)$. However, asymptotic results for $U_n$, when $U_0 \in \MM\left(S\right)$ can be easily derived from the special case  
$U_0 \in \PP\left(S\right)$. \\

\noindent
{\bf Random configuration of the urn:}
Observe that since $R$ is a stochastic kernel and $U_0 \in \PP\left(S\right)$, we have
\begin{equation}
U_n\left(S\right) = n + 1,
\label{Equ:Sum-of-weights}
\end{equation}
for all $n \geq 0$. With slight abuse of terminology, we will call the random probability measure $\frac{U_n}{n+1}$, 
as the \emph{random configuration of the urn}. In fact, 
\begin{equation}
\Pbold\left(Z_n \in A \,\Big\vert\, U_n, U_{n-1}, \cdots, U_0 \right) = \frac{U_n\left(A\right)}{n+1}, \,\,\, A \in \SS. 
\label{Equ:Choise-Mass-Function-infinite}
\end{equation}
In other words, the $n$-th random configuration of the urn is the conditional distribution of the $(n+1)$-th
selected color, given $U_0, U_1, \ldots, U_n$. \\

\noindent
{\bf Expected configuration of the urn:}
Taking expectation in equation \eqref{Equ:Sum-of-weights}, we get 
\begin{equation}
{\mathbb E}\left[U_n\right] \left(S\right) = {\mathbb E}\left[U_n \left(S\right)\right] = n + 1, 
\label{Equ:Expected-sum-of-weights}
\end{equation}
for all $n \geq 0$. Thus $\frac{{\mathbb E} \left[U_n\right]}{n+1}$ is also a probability measure on $\left(S, \SS\right)$. 
In fact, it is the distribution of $Z_n$, the $(n+1)$-th selected color. This follows by taking expectation on both sides of
equation \eqref{Equ:Choise-Mass-Function-infinite}, 
\begin{equation}
\Pbold\left(Z_n \in A\right) = \frac{{\mathbb E}\left[U_n\right] \left(A\right)}{n+1}, \,\,\, A \in \SS.
\end{equation}

\subsection{Notations}
\label{SubSec:Notations}
Most of the notations used in this paper are 
consistent with the literature on finite color urn models. However, 
we use few specific notations as well, which are given below. 

\begin{itemize}

\item[(i)] All vectors are written as row vectors unless otherwise stated. Column vectors are denoted by $x^{T},$ 
where $x$ is a row vector.

\item[(ii)] For any vector $x$, $x^2$ will denote a vector with the coordinates squared.

\item[(iii)] The standard Gaussian measure on $\Rbold^d$ will be denoted by
$\Phi_d$ with its density given by 
\[
\phi_d\left(x\right) := \frac{1}{\left(2 \pi \right)^{d/2}} \exp\left(- \frac{\| x \|^2}{2}\right), x \in \Rbold^d.
\]
For $d=1$, we will simply write $\Phi$ for the standard Gaussian measure on 
$\Rbold$ and $\phi$ for its density.
\item[(iv)] The symbol $\Rightarrow$ will denote convergence in distribution of random variables.
\item[(v)] The symbol $\stackrel{p}{\longrightarrow}$ will denote convergence in probability.
\item[(vi)] The symbol $\stackrel{w}{\longrightarrow}$ will denote the weak convergence of probability measures in $\PP\left(S\right)$.

\end{itemize}
 
\subsection{Outline of the Paper} 

The rest of the paper is organized as follows: Section \ref{Sec: Main Results} contains 
the two representation theorems, namely, Theorem \ref{Thm: Grand Representation} and
Theorem \ref{Thm: Rep Theorem}, 
which are the most important contribution of this work. 
In Section ~\ref{Sec:Asymptotic}, we derive asymptotic results for the random and expected configurations under fairly 
general conditions.  
In Section \ref{Sec: Applications}, we provide many interesting applications mainly in the context of
infinite color urn schemes. 

\section{Main Results and Their Proofs}
\label{Sec: Main Results}
In this section, we present the two main theorems of this paper. These theorems, which we call the 
\emph{Grand Representation Theorem} (Theorem ~\ref{Thm: Grand Representation}) and the
\emph{Marginal Representation Theorem} (Theorem ~\ref{Thm: Rep Theorem}) 
provide certain ``\emph{couplings}'' of the urn model with the associated Markov chain through the 
\emph{observed sequence of colors}. These results are fairly general and hold for any balance urn schemes 
with colors indexed by an \emph{arbitrary} set $S$.

\subsection{Grand Representation Theorem}
\label{SubSec:Grad-Representation}
The following theorem is a ``\emph{representation}'' of the entire sequence of colors 
$\left(Z_n\right)_{n \geq 0}$ in terms of the associated Markov chain. 
\begin{theorem}\label{Thm: Grand Representation}
Consider an urn model with colors indexed by a set $S$ endowed with a 
$\sigma$-algebra $\SS$. Let $R$ be the replacement kernel and $U_0$ be the initial configuration.
For $n \geq 0$, let $Z_n$ be the random color of the $(n+1)$-th draw. Further, for $n \geq -1$, 
let $\TT_n$ be the \emph{random recursive tree} on 
$\left(n+2\right)$ vertices labeled by $\left\{-1; 0, 1,  \cdots, n\right\}$ with $-1$ as the root. Let $\left(W_n\right)_{n \geq -1}$ be 
the \emph{branching Markov chain} on $\TT := \cup_{n \geq -1} \TT_n$ starting at the root $-1$ and at a position 
$\Delta \not\in S$ with Markov transition kernel 
$\hat{R}$ on $\left(\hat{S}, \hat{\SS}\right)$, where $\hat{S} := \left\{\Delta\right\} \cup S$ and
$\hat{\SS} := \SS \cup \left\{ A \cup \left\{\Delta \right\} \,\Big\vert\, A \in \SS\, \right\}$, is given by 
\begin{equation}
\hat{R}\left(\hat{s}, A\right) 
:= \left\{ \begin{array}{lcl}
   U_0\left(A\right)        & \mbox{\ if\ } & \hat{s} = \Delta;    \\
                            &               &                      \\
   R\left(\hat{s}, A\right) & \mbox{\ if\ } & \hat{s} \neq \Delta, \\                 
   \end{array} 
   \right.
\label{Equ:Definition-of-R-hat}   
\end{equation}
for $\hat{s} \in \left\{\Delta\right\} \cup S$ and $A \in \SS$. 
Then
\begin{equation}\label{Eq:Grand Representation}
\left(Z_n\right)_{n \geq 0} \el \left(W_n\right)_{n \geq 0}.
\end{equation}
\end{theorem}

\begin{proof}
Observe that from \eqref{Equ:Fundamental-Recurssion}, we get for $n \geq 1$, 
\begin{equation}
\frac{U_n}{n+1} = \frac{U_0}{n+1} + \sum_{k=0}^{n-1} \frac{R\left(Z_k, \cdot\right)}{n+1}.
\end{equation}
Moreover, the conditional distribution of $Z_n$ given $ Z_0, Z_1, \cdots, Z_{n-1}$ is given by the random configuration of the urn, 
namely, $\frac{U_n}{n+1}$. Now for $n \geq 0$, let $D_n$ be a \emph{discrete uniform} random variable on the set of indices 
$\left\{-1; 0, 1, \cdots, n-1\,\right\}$ such that the sequence $\left(D_n\right)_{n \geq 0}$ are independent and for each $n \geq 0$ 
the random variable $D_n$ is independent of $Z_0, Z_1, \cdots, Z_{n-1}$. Then
\begin{equation}
\left(Z_n \,\Big\vert\, Z_{n-1}, Z_{n-2}, \cdots, Z_1, Z_0\right) \ed R\left(Z_{D_n}, \cdot\right). 
\end{equation}
Further from definition we get
\begin{equation}
\left(W_n \,\Big\vert\, W_{n-1}, W_{n-2}, \cdots, W_1, W_0\right) \ed R\left(W_{D_n}, \cdot\right). 
\end{equation}
This completes the proof. 
\end{proof}

\subsection{Marginal Representation Theorem}
\label{SubSec:Marginal-Representation}
Our next result is a ``\emph{representation}'' of the marginal distribution for the randomly chosen color $Z_n$ in terms of the 
marginal distribution of the corresponding Markov chain sampled at random but independent times. As we will see from the proof it is an
immediate corollary of Theorem ~\ref{Thm: Grand Representation}. 
\begin{theorem}
\label{Thm: Rep Theorem} 
Consider an urn model with colors indexed by a set $S$ endowed with a 
$\sigma$-algebra $\SS$. Let $R$ be the replacement kernel and $U_0$ be the initial configuration.
For $n \geq 0$, let $Z_n$ be the random color of the $(n+1)$-th draw.
Let $\left(X_n\right)_{n \geq 0}$ be the associated Markov chain on $S$ with transition kernel $R$ and initial distribution $U_0$. 
Then there exists an increasing non-negative sequence of stopping 
times $\left(\tau_n\right)_{n \geq 0}$ with $\tau_0=0$, which are independent of the Markov chain $\left(X_n\right)_{n \geq 0}$, 
such that, 
\begin{eqnarray}\label{Eq: Rep}
Z_n \el X_{\tau_{n}},
\end{eqnarray} 
for any $n \geq 0$. Moreover, as $n \to \infty$,
\begin{equation}\label{Eq: SLT}
\frac{\tau_{n}}{\log n}\longrightarrow 1 \text{ a.s. } 
\end{equation} 
and 
\begin{equation}\label{Eq: CLTT}
\frac{\tau_{n}-\log n}{\sqrt{\log n}}\stackrel{d}{\longrightarrow} N\left(0,1\right).
\end{equation} 
\end{theorem}

\noindent
{\bf Remark:}
Recall that the probability mass function of $Z_n$ is $\frac{1}{n+1}\bE\left[U_n\right]$, which implies that \eqref{Eq: Rep} will be 
useful in deriving results for the expected configuration. A version of this result was obtained in Proposition 7 in \cite{BaTh2013}, 
which was restricted to the case when $\left(X_n\right)_{n \geq 0}$ is a bounded increment random walk. Here however, the result is 
for any general Markov chain $\left(X_n\right)_{n \geq 0}.$ \\

\noindent
{\bf Remark:}
It is worthwhile to note here that the following may not be necessarily true
\begin{equation}\label{Eq:Z_nnotMarkov}
\left(Z_n\right)_{n \geq 0}\el \left(X_{\tau_n}\right)_{n \geq 0}.
\end{equation} 
This is because $\left(Z_n\right)_{n \geq 0}$ is not necessarily Markov, but $\left(X_{\tau_n}\right)_{n \geq 0}$ is Markovian. 
In fact, the law of the process $\left(Z_n\right)_{n \geq 0}$ is more complicated as presented in 
Theorem ~\ref{Thm: Grand Representation}.\\

\begin{proof}
Let $(1 + \tau_n)$ be length of the unique path from the vertex $n$ to the root $-1$ in the random 
recursive tree $\TT_n$ with $n+1$ vertices, as defined in 
the Theorem ~\ref{Thm: Grand Representation}. Then from definition it follows that
\begin{eqnarray}
W_n \el X_{\tau_{n}},
\end{eqnarray} 
and thus \eqref{Eq: Rep} follows from Theorem ~\ref{Thm: Grand Representation}. 

Now, for $0 \leq j \leq n-1$, let $I_j$ be the indicator that the vertex $j$ lies on the path from the root $-1$ to the vertex $n$. 
Then by construction $\left(I_j\right)_{0 \leq j \leq n-1}$ are independent Bernoulli variables with 
$\bE\left[I_j\right] = \frac{1}{j+2}$, $0 \leq j \leq n-1$. Also, 
\begin{equation}
\tau_n = \sum_{j=0}^{n-1}  I_j.
\label{Equ:Representation-of-tau}
\end{equation}

Notice that $\Var\left(I_j\right) = \frac{1}{j+2} \left(1 - \frac{1}{j+2}\right)$, $0 \leq j \leq n-1$, thus
\begin{equation}
\bE\left[\tau_n\right] = \sum_{j=0}^{n-1} \frac{1}{j+2} \sim \log n 
\mbox{\ \ and\ \ }
\Var\left(\tau_n\right) = \sum_{j=0}^{n-1} \frac{1}{j+2} \left(1 - \frac{1}{j+2}\right) \sim \log n, 
\label{Equ:Mean-Variance-of-tau}
\end{equation}
as $n \rightarrow \infty$. So by Kronecker's Lemma (see (8.5) on page 63 of \cite{Durrett2010}) it follows that
\begin{equation}
\frac{\tau_{n}}{\log n} \longrightarrow 1 \text{ a.s.}
\label{Equ:Asymptotic-tau}
\end{equation} 
as $n \rightarrow \infty$, proving \eqref{Eq: SLT}. 
Further, \eqref{Eq: CLTT} follows by an easy application of the Lyapunov Central Limit Theorem 
(see Theorem 27.3 on page 362 of \cite{Bi95}). 
\end{proof}

\section{Weak Asymptotic of the Urn Configuration}
\label{Sec:Asymptotic}
In this section we state and prove some very general results for the asymptotic of the random and expected configurations of 
our general urn scheme $\left(U_n\right)_{n \geq 0}$, defined in Section ~\ref{Subsec: Model}. These results will be proved using the
two representations theorems given in Section ~\ref{Sec: Main Results}. We start by establishing an asymptotic result for the
branching Markov chain $\left(W_n\right)_{n \geq -1}$ as defined in Theorem ~\ref{Thm: Grand Representation}. 
For this and the later sections, we assume that $\PP\left(S\right)$ is endowed with the topology of \emph{weak convergence} and any 
limit statement in $\PP\left(S\right)$ is with respect to the topology of weak convergence.

\subsection{Asymptotic of Branching Markov Chain on Random Recursive Tree}
\label{SubSec:Asymptotic-BRW-on-RRT}
Let us first recall that $\left(W_n\right)_{n \geq -1}$ is defined as
the \emph{branching Markov chain} on the (infinite) random recursive tree
$\TT := \cup_{n \geq -1} \TT_n$, starting at the root $-1$ and at a position 
$\Delta \not\in S$ with Markov transition kernel 
$\hat{R}$ on $\left(\hat{S}, \hat{\SS}\right)$ given in ~\eqref{Equ:Definition-of-R-hat}. 
Define $\GG_n := \sigma\left(W_0, W_1, \cdots, W_{n-1}\right)$, $n \geq 0$. 
Let $Q_n$ be a version of the regular conditional distribution of $W_n$ given $\GG_n$. Note that $Q_n$ exists and is almost surely unique
and proper, as $S$ is a Polish space and $\SS$ is the corresponding Borel $\sigma$-algebra. Further, 
let $q_n$ be a version of the regular conditional distribution of $W_n$ given $\GG_n$ and $\TT_n$. Once again
$q_n$ exists and is almost surely unique and proper. It is worth noting here that 
for any $A \in \SS$, $\bE\left[ q_n\left(A\right) \,\Big\vert\, \GG_n \,\right] = Q_n\left(A\right)$. 

Further, recall $1 + \tau_n$ is defined to be the length of the unique path from the vertex $n$ to the root $-1$ in the random 
recursive tree $\TT$. 

Also recall that 
$\left(X_n\right)_{n \geq 0}$ denotes a Markov chain with state space $S$, transition kernel $R$ and starting distribution $U_0$. 

We now make the following assumption:
\begin{itemize}

\item[{\bf (A)}] There exists a (non-random) probability $\Lambda$ on $\left(\Rbold^d, \BB_{{\mathbb R}^d}\right)$ 
            and a vector $\bv \in \Rbold^d$, 
            and two functions $a: \Rbold_+ \rightarrow \Rbold$ and 
            $b: \Rbold_+ \rightarrow \Rbold_+$, such that, for any starting distribution $U_0$,
            \begin{equation}
            \frac{X_n - a\left(n\right) \bv}{b\left(n\right)} \cd \Lambda.
            \label{Equ:Assumption}
            \end{equation}

\end{itemize}

\begin{theorem}
\label{Thm:Asymptotic-BRW-on-RRT}
Suppose assumption {\bf (A)} holds. 
Let $q_n^{\text{cs}}$ be the conditional distribution of $\frac{W_n - a\left(\tau_n\right) \bv}{b\left(\tau_n\right)}$ given
           $\GG_n$ and $\TT_n$, that is, a scaled and centered version of $q_n$ with centering $a\left(\tau_n\right) \bv$ and scaling
           $b\left(\tau_n\right)$. Then as $n \rightarrow \infty$, 
           \begin{equation}
           \label{Equ:Convergence-of-qncs}
           q_{n}^{cs}\stackrel{p}{\longrightarrow} \Lambda \mbox{\ in\ } \PP\left(\Rbold^d\right). 
           \end{equation}
Moreover, let $Q_n^{\text{cs}}$ is the conditional distribution of $\frac{W_n - a\left(\log n\right) \bv}{b\left(\log n\right)}$ given
$\GG_n$, that is, a scaled and centered version of $Q_n$ with centering by $a\left(\log n\right) \bv$ and scaling
by $b\left(\log n\right)$, then
\begin{itemize}
           
\item[(a)] If $a \equiv 0$ and $b \equiv 1$, then 
           \begin{equation}
           \label{Equ:Convergence-of-Qn}
           Q_{n}^{\text{cs}} = Q_{n}\stackrel{p}{\longrightarrow} \Lambda \mbox{\ in\ } \PP\left(S\right). 
           \end{equation}
           
\item[(b)] Suppose $a = 0$  and $b$ is regularly varying function, then
           \begin{equation}
           Q_{n}^{cs}\stackrel{p}{\longrightarrow} \Lambda \mbox{\ in\ } \PP\left(\Rbold^d\right).
           \label{Equ:Convergence-of-Qncs-a0}
           \end{equation}      

\item[(c)] Suppose $a$ is differentiable and $\displaystyle{\lim_{x \rightarrow \infty} a'\left(x\right) = \tilde{a} < \infty}$. Also  
           assume $b$ is regularly varying
           and $\displaystyle{\lim_{x \rightarrow \infty} \frac{\sqrt{x}}{b\left(x\right)} = \tilde{b} < \infty}$ then
           \begin{equation}
           Q_{n}^{cs}\stackrel{p}{\longrightarrow} \Xi \mbox{\ in\ } \PP\left(\Rbold^d\right),
           \label{Equ:Convergence-of-Qncs-Xi}
           \end{equation}
           where $\Xi$ is $\Lambda$ if $\tilde{a} = 0$ or $\tilde{b} = 0$, otherwise, it is given by the convolution of 
           $\Lambda$ and $\mbox{Normal}\left(0, \tilde{a}^2 \tilde{b}^2\right) \bv$.
           
\end{itemize}
\end{theorem}

\begin{proof}
Let $\rho$ be a metric on $\PP\left(S\right)$ which metrize the weak convergence topology on it. 
Denote by $L_n$ the distribution of $\frac{X_n - a\left(n\right)}{b\left(n\right)}$. Under assumption {\bf (A)}, we have
\begin{equation}
\rho\left(L_n, \Lambda\right) \longrightarrow 0, 
\end{equation}
as $n \rightarrow \infty$. 

Now, fix $\eps > 0$ and
find $H > 0$ so large that $\rho\left(L_h, \Lambda\right) < \eps$, for any $h > H$. 
Find $N \geq 1$ so large that $\frac{\left(\log \left(n+2\right)\right)^H}{n+2} < \eps$ for all $n \geq N$.

Let $S_n^H$ be the set of vertices of the random recursive tree ${\mathcal T}_n$ up to depth $H$. 
Then $\bE\left[ \left\vert S_n^H \right\vert \right] = \left(\log \left(n+2\right)\right)^H + o(1)$.

Recall that $D_n$ denotes the vertex at which $n$-th vertex joins in the random recursive tree $\TT_{n}$. Then by construction
\begin{equation}
\bP\left( \rho\left(q_n^{\text{cs}}, \Lambda\right) > \eps\right) 
\leq \bP\left( D_n \in S_n^H\right) = \frac{\bE\left[S_n^H \right]}{n+2} < \eps.
\end{equation}
This completes the proof of the first part. 

For proof of part (a), it is enough to observe that under the assumptions of $a = 0$ and $b = 1$, 
$Q_n^{\text{cs}} = \bE\left[ q_n^{\text{cs}} \,\Big\vert\, \GG_n\,\right]$ almost surely. 

For part (b), we first observe that if $b$ is regularly varying, then by Karamata's Characterization Theorem \cite{GaSe_73} and
equation ~\eqref{Equ:Asymptotic-tau}
we can show that 
\begin{equation}
\frac{b\left(\log n\right)}{b\left(\tau_n\right)} \stackrel{p}{\longrightarrow} 1.
\end{equation}
Further by the delta-method \cite{Durrett2010}, under assumptions in part (c), 
\begin{equation}
\frac{a\left(\tau_n\right) - a\left(\log n\right)}{\sqrt{\log n}} 
\Rightarrow 
\left\{
\begin{array}{ll}
\mbox{Normal}\left(0, \tilde{a}\right) & \mbox{if\ } \tilde{a} \neq 0; \\
\delta_0                               & \mbox{otherwise}.
\end{array}
\right.
\end{equation}
Finally, we note that
\begin{equation}
\frac{W_n - a\left(\tau_n\right) \bv}{b\left(\tau_n\right)}
= 
\frac{b\left(\log n\right)}{b\left(\tau_n\right)} \, 
\left(
\frac{W_n - a\left(\log n\right) \bv}{b\left(\log n\right)}
+ 
\frac{\sqrt{\log n}}{b\left(\log n\right)} \, \frac{a\left(\tau_n\right) - a\left(\log n\right)}{\sqrt{\log n}}
\right).
\end{equation}

\end{proof}

\subsection{Asymptotic of the Random Configuration of the Urn}
\label{SubSec:Asymptotic-Random-Urn}
Define $\FF_n := \sigma\left(Z_0, Z_1, \cdots, Z_n\right)$, $n \geq 0$. 
Let $P_n$ be a version of the regular conditional distribution of $Z_n$ given $\FF_n$. 
Note by construction $P_n = \frac{U_n}{n+1}$ almost surely. 
The following result is an immediate corollary of the Theorem ~\ref{Thm: Grand Representation} and 
Theorem ~\ref{Thm:Asymptotic-BRW-on-RRT}.
\begin{theorem}
\label{Thm:Asymptotic-Random-Urn}
Suppose assumption {\bf (A)} holds. 
Let $P_n^{\text{cs}}$ is the conditional distribution of $\frac{Z_n - a\left(\log n\right) \bv}{b\left(\log n\right)}$ given
$\FF_n$, that is, a scaled and centered version of $P_n$ with centering by $a\left(\log n\right) \bv$ and scaling
by $b\left(\log n\right)$, then
\begin{itemize}
           
\item[(a)] If $a = 0$ and $b = 1$, then 
           \begin{equation}
           \label{Equ:Convergence-of-Pn}
           P_{n}^{\text{cs}} = P_{n}\stackrel{p}{\longrightarrow} \Lambda \mbox{\ in\ } \PP\left(S\right). 
           \end{equation}
           
\item[(b)] If conditions of the Part (b) of the Theorem ~\ref{Thm:Asymptotic-BRW-on-RRT} hold, then
           \begin{equation}
           P_{n}^{cs}\stackrel{p}{\longrightarrow} \Lambda \mbox{\ in\ } \PP\left(\Rbold^d\right).
           \label{Equ:Convergence-of-Pncs-a0}
           \end{equation}
           
\item[(c)] If conditions of the Part (c) of the Theorem ~\ref{Thm:Asymptotic-BRW-on-RRT} hold, then
           \begin{equation}
           P_{n}^{cs}\stackrel{p}{\longrightarrow} \Xi \mbox{\ in\ } \PP\left(\Rbold^d\right),
           \label{Equ:Convergence-of-Pncs}
           \end{equation}
           where $\Xi$ is $\Lambda$ if $\tilde{a} = 0$ or $\tilde{b} = 0$, otherwise, it is given by the convolution of 
           $\Lambda$ and $\mbox{Normal}\left(0, \tilde{a}^2 \tilde{b}^2\right) \bv$.
           
\end{itemize}
\end{theorem}

\subsection{Asymptotic of the Expected Configuration of the Urn}
\label{SubSec:Asymptotic-Expected-Urn}
Recall that $\bE\left[P_n\right] = \frac{\bE\left[U_n\right]}{n+1}$ is the marginal distribution of $Z_n$. The  
following result is an immediate corollary of the Theorem ~\ref{Thm:Asymptotic-Random-Urn}. 
\begin{theorem}
\label{Thm:Asymptotic-Expected-Urn}
Suppose assumption {\bf (A)} holds, then
\begin{itemize}
           
\item[(a)] If $a = 0$ and $b = 1$, then 
           \begin{equation}
           \label{Equ:Convergence-of-Zn}
           Z_n \Rightarrow \Lambda. 
           \end{equation}
           
\item[(b)] If conditions of the Part (b) of the Theorem ~\ref{Thm:Asymptotic-BRW-on-RRT} hold, then
           \begin{equation}
           \frac{Z_n - a\left(\log n\right) \bv}{b\left(\log n\right)} \Rightarrow \Lambda,
           \label{Equ:Convergence-of-Zncs-a0}
           \end{equation}         
           
\item[(c)] If conditions of the Part (c) of the Theorem ~\ref{Thm:Asymptotic-BRW-on-RRT} hold, then
           \begin{equation}
           \frac{Z_n - a\left(\log n\right) \bv}{b\left(\log n\right)} \Rightarrow \Xi,
           \label{Equ:Convergence-of-Zncs}
           \end{equation}
           where $\Xi$ is $\Lambda$ if $\tilde{a} = 0$ or $\tilde{b} = 0$, otherwise, it is given by the convolution of 
           $\Lambda$ and $\mbox{Normal}\left(0, \tilde{a}^2 \tilde{b}^2\right) \bv$.
           
\end{itemize}
\end{theorem}

\begin{proof}
The result follows from the Theorem ~\ref{Thm:Asymptotic-Random-Urn} by taking expectation 
and noting the fact that centering and scaling are non-random in all cases. 
\end{proof}

\section{Applications}
\label{Sec: Applications}
In this section we discuss several applications of the representation theorems (Theorem ~\ref{Thm: Grand Representation} 
and Theorem ~\ref{Thm: Rep Theorem}). Essentially all the results stated here are proved using the two general asymptotic 
results, namely, Theorem ~\ref{Thm:Asymptotic-Random-Urn} and Theorem ~\ref{Thm:Asymptotic-Expected-Urn}, given in the previous 
section. 

\subsection{$S$ is Countable and $R$ is Ergodic}
\label{SubSec:Classical}
Suppose the indexing set of colors $S$ is either finite or countably infinite and we endow $S$ with the $sigma$-algebra $\SS$, which is 
the power set $\wp\left(S\right)$. In this case, we can view the Markov transition kernel $R$ as a matrix and it is then called
the \emph{replacement matrix}. For $S$ finite, it is the classical case. If we assume that $R$ is ergodic, that is, 
assumption {\bf (A)} holds with $a = 0$ and $b = 1$, then from Theorem ~\ref{Thm:Asymptotic-Random-Urn}(a) we get the following 
result. 
\begin{theorem}
\label{Thm:Classical}
Suppose $S$ is countable, $\SS = \wp\left(S\right)$, $R$ is ergodic with stationary distribution $\pi$ on $S$. Then
as $n \rightarrow \infty$, 
\begin{equation}
\frac{U_n}{n+1} \stackrel{p}{\longrightarrow} \pi \mbox{\ in\ } \PP\left(S\right).
\label{Equ:Classical-Asumptotic}
\end{equation}
In particular, 
\begin{equation}
\frac{\bE\left[U_n\right]}{n+1} \stackrel{w}{\longrightarrow} \pi \mbox{\ in\ } \PP\left(S\right),
\label{Equ:Classical-Expected-Asumptotic}
\end{equation}
as $n \rightarrow \infty$. 
\end{theorem}

If $S$ is finite then using either matrix algebra techniques or multi-type branching process techniques, it is known 
\cite{Gouet, Svante1} that stronger result holds. In fact, under even weaker assumption of only \emph{irreducibility} of the chain, 
the convergence in probability in \eqref{Equ:Classical-Asumptotic} can be replaced by almost sure convergence. 
We believe that in general for $S$ countable, under ergodicity assumption almost sure convergence should hold. We here note that
as soon as $S$ is infinite, the classical techniques such as matrix algebra methods using 
\emph{Perron-Frobenius theory} of matrices with positive entries \cite{Sene06}
and \emph{Jordan Decomposition} of finite dimensional matrices \cite{Curtis93}, or martingale approach using 
embedding to multi-type branching processes, which have been extensively used in classical urn model 
literature \cite{AthKar68, Gouet, Svante1, Svante2, BaiHu05, maulik1, DasMau11}; fails to derive any result. 
We are hopeful that our novel and fairly probabilistic approach, namely, the \emph{Grand and Marginal representation Theorems}
should yield the classical result. Unfortunately, we have been unable to derive it so far.

\subsection{$S$ is Countable and $R$ is Block Diagonal}
\label{SubSec:Block-Diagonal}
Similar to the previous section, 
suppose the indexing set of colors $S$ is either finite or countably infinite and we endow $S$ with the $sigma$-algebra $\SS$, which is 
the power set $\wp\left(S\right)$. As in the previous case, we view the Markov transition kernel $R$ as a matrix. 
Suppose the replacement matrix $R$, can be decomposed in the following manner. Let the indexing set of colors be partitioned as
$\displaystyle{S = \mathop{\cup}\limits_{i \in \II} C_i}$, where $\II$ is a countable set, and $C_i$ is countable for all $i \in \II$.
We endow $\II$ with its power set as a $\sigma$-algebra on it and each $C_i$ is also endowed with 
its power set as the $\sigma$-algebra on it. 

Now let, $\phi: S \rightarrow \II$ be the ``projection'' map, which maps $s \mapsto i$, where $i$ is the unique element of $\II$, 
such that, $s \in C_i$.

Now, suppose for every 
$i \in \II$ and $s \in C_i$, the kernel $R\left(s, \cdot\right)$ is a probability measure supported only on $C_i$, that is,
\begin{equation}\label{Eq:projection_R}
R\left(s, C_i\right) = \left\{
                       \begin{array}{ll}
                       1 & \mbox{if\ } s \in C_i\\
                       0 & \mbox{otherwise}.
                       \end{array}
                       \right.
\end{equation}
As each $C_i$ is countable, $R$ on $C_i$ can be
realized as a (possibly infinite) matrix $R_{ii}$ indexed by the colors in $C_i$. 

Note that if $S$ is finite then $R$ is essentially 
a reducible matrix with diagonal blocks, which can be presented as
\[
R = \left( 
\begin{array}{ccccc}
R_{11} & 0      & 0      & \cdots & 0      \\
0      & R_{22} & 0      & \cdots & 0      \\
0      & 0      & R_{33} & \cdots & 0      \\
\vdots & \vdots & \vdots & \ddots & \vdots \\
0      & 0      & 0      & \cdots & R_{kk} \\
\end{array} 
\right),
\]

We further assume that the for all $i \in \II$, the
kernel/replacement matrix, $R_{ii}$ restricted to its ``block'' $C_i$, is ergodic with stationary distribution,
$\pi_i$.

\begin{theorem}
\label{Thm:Block_diagonal_form}
Consider an urn model with colors indexed by a set $S$ and replacement kernel $R$ as in \eqref{Eq:projection_R}.  
Then for every initial configuration $U_0$, as $n \to \infty$,
\begin{equation}
\label{Eq:a.s_conv_block_diagonal}
\frac{U_n}{n+1} \stackrel{p}{\longrightarrow} \Pi \mbox{\ \ in\ \ } \PP\left(S\right), 
\end{equation} 
where $\Pi$ is a random probability measure on $\left(S, \SS\right)$ given by 
\begin{equation}
\Pi\left(A\right) =  \sum_{i \in {\mathcal I}}  \pi_i\left(A \cap C_i\right) \, \nu_i, \,\,\, A \in \SS, 
\end{equation}
and $\nu$ has \emph{Ferguson Distribution} on the countable set $\II$ with parameter $U_0 \circ \phi^{-1}$.
\end{theorem}

\begin{proof}
Let us denote by $c_i:=\textstyle \sum_{v \in C_i} U_{0,v},$ and $T_{n,i}:= \sum_{v \in C_i} U_{n,v}$ for each $i \geq 1.$ 
It is easy to check that for each $i \geq 1,$ the sequence $\left(\frac{T_{n,i}}{n+1}\right)_{n \geq 0 }$ is non-negative a.s.\ 
convergent martingale, see \cite{Gouet} for details. Consider each $C_i$ to be a super color for $i \geq 1$. 
Then $T_n:=\left(T_{n,i}\right)_{i \geq 1}$ corresponds to the configuration of a classical P\'olya urn model, 
with initial configuration $T_0=\left(c_i\right)_{i \geq 1}.$ It is worthwhile to note here that $\frac{T_n}{n+1}$ is a 
probability measure on $\II$. Therefore, 
from \cite{BlackMac73}, as $n \to \infty$
\begin{equation}\label{Eq:Ferguson_dist_conv_a.s}
\frac{T_n}{n+1} \longrightarrow \nu \text{ a.s.,}
\end{equation} where $\nu $ is a random measure on $\II$ having Ferguson Distribution with parameter $U_0 \circ \phi^{-1}.$ 

Define $N_{n,i}:=\textstyle \sum_{k=1}^n \bone_{Z_k \in C_i}.$ Then it is obvious that $T_{n,i}=c_i+N_{n,i}.$ 
Writing $U_n=\left(U_{n,C_i}\right)_{i \geq 1},$ we observe that for each $i\geq 1,$ $\left(U_{N_{n,i}, C_i}\right)$ is 
the configuration of an urn with initial configuration $\frac{U_{0,C_i}}{c_i}$ and replacement matrix $R_{ii}.$ 
Since from ~\eqref{Eq:Ferguson_dist_conv_a.s} we know that $N_{n,i} \longrightarrow \infty \text{ a.s.}$ as $n \to \infty,$ it 
follows from Theorem ~\ref{Thm:Classical}
that as $n \to \infty,$
\begin{equation}\label{Eq:a.s.convg_block}
\frac{U_{N_{n,i},C_i}}{T_{n,i}} \cp \pi_i \mbox{\ \ in\ \ } \PP\left(S\right),
\end{equation} 
This implies that 
\begin{equation}\label{Eq:Fergusion_stationary_a.s_convg}
\frac{U_{n,C_i}}{n+1}=\frac{U_{N_{n,i}, C_i}}{n+1}= \frac{U_{N_{n,i},C_i}}{T_{n,i}}\frac{T_{n,i}}{n+1}
\cp \nu(C_i)\pi_i
\end{equation}where $\nu$ as in \eqref{Eq:Ferguson_dist_conv_a.s}.  This completes the proof. 
\end{proof}

\subsection{Urn Models Associated with Random Walks on $\bZ^d$}
\label{SubSec: GeneralRW}
It this section we consider urn models associated with random walks on $\Zbold^d$. 
These models were first introduced in \cite{BaTh2013}, where only bounded increment walks were considered. 
In general, here we take $S = \Zbold^d$ for some $d \geq 1$ and $\SS$ will be taken as the power set of $\Zbold^d$. 
The kernel $R$ can be viewed as an infinite dimensional matrix index by the set of colors $\Zbold^d$, given by
\begin{equation}
R\left(u,v\right)=\bp\left(v-u\right), \,\,\, u,v \in \Zbold^d,
\label{Equ:definition-of-R-in-RW}
\end{equation} 
where $\bp$ is the distribution on $\Zbold^d$ of the independent increments of the walk. 

\subsubsection{Finite Variance Walks}
\label{SubSubSec:Finite-Variance}
Suppose $\bp$ has finite second moment, leading to a random walk with finite variance. 
Following theorem is a generalization of the results derived in \cite{BaTh2013}.
\begin{theorem}
\label{Thm:Finite-Varinace-Walk}
Consider an infinite color urn model with colors indexed by $S =\Zbold^d$, and kernel $R$ as given above. Suppose the
starting configuration is $U_0$. Then there exist $\bmu \in \Rbold^d$ and a positive definite matrix 
$\varSigma_{d \times d}$, such that, if we define, 
\[
P_n^{\text{cs}}\left(A\right) :=\frac{U_n}{n+1}\left(\sqrt{\log n}A\varSigma^{1/2}  +  \bmu \log n\right), 
\,\,\, A \in \BB_{{\mathbb R}^d}, 
\]
where
\[
xA\varSigma^{1/2}: =\{xy \varSigma^{1/2}\colon y \in A\},
\]
then, as $n \to \infty $,
\begin{equation}
P_{n}^{cs} \stackrel{p}{\longrightarrow} \Phi_{d} \mbox{\ in\ } \PP\left(\Rbold^d\right).
\label{Equ:Asymptotic-Random-Urn-RW}
\end{equation}
In particular, 
\begin{equation}
\frac{Z_{n}-\bmu\log n} {\sqrt{\log n}} \Rightarrow \mbox{Normal}_{d}(0,\varSigma),
\label{Equ:Asymptotic-Expected-Urn-RW}
\end{equation}
as $n \rightarrow \infty$
\end{theorem} 

\begin{proof}
Let $X_n$ be the position of the random walk starting with $X_0 \sim U_0$ and independent increments with distribution given by 
$\bp$. From classical Central Limit Theorem \cite{Durrett2010}, we get that 
\begin{equation}
\frac{X_n - n \bmu}{\sqrt{n}} \Rightarrow \mbox{Normal}_d\left(\bzero, \varSigma - \bmu \bmu^T\right),
\end{equation}
where $\bmu$ is the mean of the increment distribution and $\varSigma$ is the second moment. Thus assumption {\bf (A)} holds,
with  $\bv = \bmu$, $a\left(n\right) = n$, $b\left(n\right) = \sqrt{n}$ and 
$\Lambda = \mbox{Normal}_d\left(\bzero, \varSigma - \bmu \bmu^T\right)$. 

We observe that the assumptions in Part (c) of Theorem ~\ref{Thm:Asymptotic-BRW-on-RRT} holds with $\tilde{a} = 1$ and 
$\tilde{b} = 1$. 
This completes the proof of ~\eqref{Equ:Asymptotic-Random-Urn-RW}, by observing that 
$\Xi = \mbox{Normal}_{d}(0,\varSigma)$. 

Finally, ~\eqref{Equ:Asymptotic-Expected-Urn-RW} follows from Theorem ~\ref{Thm:Asymptotic-Expected-Urn}(c).
\end{proof}

\subsubsection{Symmetric $\alpha$-Stable Walks}
\label{SubSubSec:Infinite-Variance}
As an example of an infinite variance case, which could not have been done by the techniques derived in \cite{BaTh2013}, we 
consider the case $\bp$ is a \emph{symmetric $\alpha$-stable} distribution on $\Zbold$, for sake of completeness 
we provide here the definition. 
\begin{definition}
\label{Def: SalphaS}
A distribution $\bp$ is said to have a symmetric $\alpha$-stable distribution, denoted by $S \alpha S$, 
with $0< \alpha \leq 2$, if for any $t \in \bR$,
\begin{equation}
\bE \left[e^{itY}\right]=\exp\left(-\sigma^{\alpha }\lvert t\rvert^{\alpha }\right),
\end{equation}
for some $\sigma > 0$, where $Y \sim \bp$. 
\end{definition}
We further assume that $\bp$ has the following properties. If $Y \sim \bp$, then
\begin{enumerate}
\item[(i)] There exists $0< \alpha \leq 2$, and a slowly varying function $L \left(\cdot\right)$, such that 
\begin{equation}\label{Eq: heavy tail}
n^{\alpha} \bP \left(\lvert Y \rvert >n\right)=L\left(n\right), \text{ for all } n \in \bN; \mbox{\ and}
\end{equation}
\item[(ii)]
\begin{equation}\label{Eq: SymmetryY}
\displaystyle\lim_{n \to \infty}\frac{\bP \left( Y >n\right)}{\bP \left(\lvert Y \rvert >n\right)}=\frac{1}{2}.
\end{equation}
\end{enumerate}

\begin{theorem}
\label{Thm: SalphaSurn}
Consider an infinite color urn model with colors indexed by $S =\Zbold^d$, and kernel $R$ defined by 
~\eqref{Equ:definition-of-R-in-RW}, where $\bp$ is as defined above. 
Suppose the starting configuration is $U_0$. Then there exist sequences $\left(c_n\right)_{n \geq 1}$ 
and $\left(s_n\right)_{n \geq 1}$, 
and a $S \alpha S$-distribution $\Lambda$, where $\alpha$ is as in \eqref{Eq: heavy tail}, such that, 
if we define, 
$P_n^{\text{cs}}$ as the conditional distribution of $\frac{Z_n - c_n}{s_n}$ given $\FF_n$, 
then, as $n \to \infty $,
\begin{equation}
P_{n}^{cs} \stackrel{p}{\longrightarrow} \Lambda \mbox{\ in\ } \PP\left(\Rbold^d\right).
\label{Equ:Asymptotic-Random-Urn-RW-SalphaS}
\end{equation}
In particular, 
\begin{equation}
\frac{Z_{n}- c_n} {s_n} \Rightarrow \Lambda,
\label{Equ:Asymptotic-Expected-Urn-RW-SalphaS}
\end{equation}
as $n \rightarrow \infty$. 

Moreover, we may choose $s_n=\left(\log n\right)^{\frac{1}{\alpha}} h\left(\log n\right)$, $n \geq 1$, 
where $h\left(\cdot\right)$ is a suitable slowly varying function, and $c_n = 0, \,\,\, \forall \,\, n \geq 1$, when
$0 < \alpha \leq 1$; and $c_n = \bE \left[Y\right] \, \log n$, $n \geq 1$, with $Y \sim \bp$, for 
$1 < \alpha \leq 2$. 
\end{theorem}

\begin{proof}
Let $X_n$ be the position of the random walk starting with $X_0 \sim U_0$ and independent increments with distribution given by 
$\bp$. From \cite{GaSe_73}, we get that 
\begin{equation}
\frac{X_n - a\left(n\right) \bv}{b\left(n\right)} \Rightarrow \Lambda,
\end{equation}
where $\Lambda$ is a $S \alpha S$-distribution with $\alpha$ is as in \eqref{Eq: heavy tail}. Moreover, we may choose, 
\[
a\left(x\right) = \left\{ 
                  \begin{array}{ll}
                  0 & \mbox{if\ } 0 \leq \alpha \leq 1; \\
                  x & \mbox{otherwise},
                  \end{array}
                  \right.
\]
for $x \geq 0$, and
\[
\bv  = \left\{ 
      \begin{array}{ll}
      0 & \mbox{if\ } 0 \leq \alpha \leq 1; \\
      \bE\left[Y\right] & \mbox{otherwise},
      \end{array}
      \right.
\]
where $Y \sim \bp$, and
\[
b\left(x\right) = x^{\frac{1}{\alpha}} h\left(x\right), \,\,\, x \in \Rbold_+, 
\]
for some slowly varying function $h: \Rbold_+ \rightarrow \Rbold_+$. This completes the proof by using 
Theorem ~\ref{Thm:Asymptotic-Random-Urn}(b) \& (c) and Theorem ~\ref{Thm:Asymptotic-Expected-Urn}(b) \& (c), by observing that 
in this case $\tilde{b} = 0$. 
\end{proof}

\subsection{Urn Models Associated with Periodic Random Walk on $\Rbold^d$}
\label{Subsec: PerRW}
\begin{figure}[H]
\centering
\includegraphics [scale=0.45] {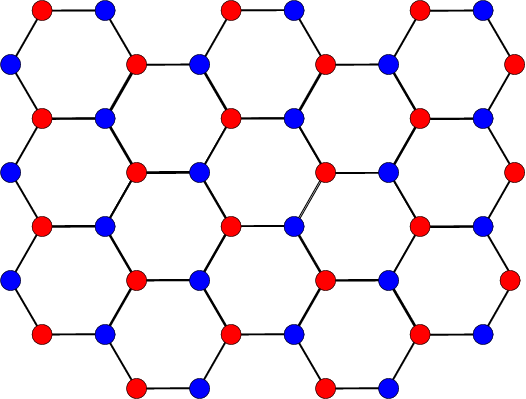}
\caption{Hexagonal Lattice}
\label{Fig:Tri}
\end{figure}
Let $\mathbb{H}=\left(V,E\right)$ be the hexagonal lattice in $\mathbb{R}^{2}$ [see Figure \ref{Fig:Tri}]. The vertex set 
can easily be partitioned into two non-empty parts, $V=V_{1}\cup V_{2}$, where $V_{1}\text{ and }V_{2}$ are disjoint, and the random walk 
$\Hbold$ is then a \emph{periodic} chain. If the replacement kernel be denoted by $R$, then corresponding urn scheme with colors
indexed by $\Hbold$, is not covered by the earlier stated Theorem ~\ref{Thm:Finite-Varinace-Walk}. For studying such cases, we
consider the following slightly more general type of random walk on $\Rbold^d$. 

Let $\{Y_j(i), 1\leq i\leq k,\mbox{ }j\geq 1\}$ be a collection of independent random $d$-dimensional vectors, such that, for each fixed $i\in \{1,2,\ldots k\}, \mbox{ }\left(Y_j(i)\right)_{j\geq 1}$ are i.i.d. We further assume that for each fixed $1\leq i\leq k,$ there exists a finite non-empty set $B_{i}\subset \bR^d$, such that,
$\bP\left(Y_1(i)\in B_{i}\right)=1$, and $B_i\cap B_j = \emptyset$, for any $1 \leq i,j\leq k$. That is, for each $ i \in \{1,2,\ldots,k\}$, we assume that the law of $Y_1(i)$ is bounded. For $1 \leq i\leq k,$ we shall write 
 \begin{equation}\label{Eq: MomentsPerRW}
\begin{array}{rcl}
\bmu(i) &:=&\bE \left[Y_{1}(i)\right],\\
\overline{\bmu}&:=&\frac{1}{k}\displaystyle \sum_{i=1}^{k}\bmu(i),\\
\varSigma(i)&:=& \bE \left[Y^T_{1}(i) Y_{1}(i)\right].
\end{array}
\end{equation}
We further assume that $\varSigma(i)$ is positive definite, for each $1\leq i \leq k.$ Let us denote by $\varSigma^{1/2}(i)$
the unique \emph{positive definite square root} of $\varSigma(i).$ Note that, then 
$\overline{\varSigma}=\frac{1}{k}\textstyle \sum_{i=1}^{k}\varSigma(i)$ is also \emph{positive definite}. We denote by 
$\overline{\varSigma}^{1/2},$ the unique \emph{positive definite square root} of $\overline{\varSigma}$.

For $n=mk+r$, where $m\in \bN \cup \{0\}$, and $0\leq r<k$, let 
\[X_{n}=X_{mk}+Y_{m+1}(1)+Y_{m+1}(2)+\ldots +Y_{m+1}(r+1),   
\]be the $k$-periodic random walk with increments $\{Y_j(i), 1\leq i\leq k,\mbox{ }j\geq 1\}$.

In the remainder of this subsection, we will consider an urn model $\left(U_n\right)_{n \geq 0}$, with colors indexed by $S = \Rbold^d$, 
starting at some distribution $U_0$ on $\Rbold^d$ and with a replacement kernel $R$ associated with a periodic random walk with periodic
increments as given above. 

\begin{theorem}
\label{Thm: PeriodicRw}
Consider an infinite color urn model with colors indexed by $S =\Rbold^d$, and kernel $R$ as given above. Suppose the
starting configuration is $U_0$. If we define, 
\[
P_n^{\text{cs}}\left(A\right) :=\frac{U_n}{n+1}\left(\sqrt{\log n}A \overline{\varSigma}^{1/2}  +  \overline{\bmu} \log n\right), 
\,\,\, A \in \BB_{{\mathbb R}^d}, 
\]
where
\[
xA\varSigma^{1/2}: =\{xy \varSigma^{1/2}\colon y \in A\},
\]
then, as $n \to \infty $,
\begin{equation}
P_{n}^{cs} \stackrel{p}{\longrightarrow} \Phi_{d} \mbox{\ in\ } \PP\left(\Rbold^d\right).
\label{Equ:Asymptotic-Random-Urn-Periodic-RW}
\end{equation}
In particular, 
\begin{equation}
\frac{Z_{n}-\overline{\bmu} \, \log n} {\sqrt{\log n}} \Rightarrow \mbox{Normal}_{d}(0,\overline{\varSigma}),
\label{Equ:Asymptotic-Expected-Urn-Periodic-RW}
\end{equation}
as $n \rightarrow \infty$
\end{theorem}

\begin{proof}
We first note that by Theorem ~\ref{Thm:Asymptotic-Random-Urn}(c) and Theorem ~\ref{Thm:Asymptotic-Expected-Urn}(c), it is
enough to show that 
\begin{equation}
\label{Eq: CITwith mean}
\frac{X_{n}- n \overline{\bmu}}{\sqrt{n}}\Rightarrow N_{d}\left(0, \overline{D}\right),
\end{equation}
where $\overline{D}=\frac{1}{k} \textstyle \sum_{i=1}^{k} \Var\left(Y_1\left(i\right)\right)$. This follows from
standard application of i.i.d. Central Limit Theorem \cite{Durrett2010}. 
\end{proof}

As an application of the Theorem ~\ref{Thm: PeriodicRw}, we now consider our starting example of the random walk on hexagonal lattice. 
Let $\mathbb{H}=\left(V,E\right)$ be the hexagonal lattice in $\mathbb{R}^{2}$ [see Figure \ref{Fig:Tri}]. The vertex set 
$V=V_{1}\cup V_{2}$, where $V_{1}\text{ and }V_{2}$ are disjoint. $V_{1}\text{ and } V_{2}$ are defined as follows:
\[V_{1,1}:=\left\{1,\omega,\omega^{2}\right\}, \text{ where } \omega \text{ is a complex cube root of unity,} 
\] and 
\[ V_{2,1}:=\left\{v+1, v+\omega, v+\omega^{2} \colon v \in V_{1,1}\right\}.
\]
For any $n \geq 2$, 
\[V_{1,n}:=\left\{v-1, v-\omega, v-\omega^{2}\colon v \in V_{2,n-1}\right\},
\] and 
\[V_{2,n}=\left\{v+1, v+\omega, v+\omega^{2} \colon  
v \in V_{1,n}\right\}.
\] Finally, $V_{1}=\textstyle \cup_{j\geq 1}V_{1,j}$ and $V_{2}=\textstyle \cup _{j\geq 1}V_{2,j}$. For any pair of vertices $v,w \in V$, we draw an edge between them, if and only if, either of the following two cases occur:
\begin{enumerate}[(i)]
\item $v \in V_1\text{ and } w \in V_2$ and $w=v+u$ for some $u\in \{1, \omega, \omega^2\}$, or 
\item $v \in V_2\text{ and } w \in V_1$ and $w=v+u$ for some $u \in \{-1, -\omega, -\omega^2\}$. 
\end{enumerate}

To define the random walk on $\mathbb{H}$, let us consider $\left\{Y_{j}(i)\colon i=1,2, \text{ }j \geq 1\right\}$ to be a sequence of 
independent random vectors such that $\left(Y_{j}(i)\right)_{j\geq 1}$ are i.i.d for every fixed $i=1,2$. 
Let $Y_{1}(1) \sim \mathrm{Unif}\left\{1,\omega,\omega^{2}\right\},$ and 
$Y_{1}(2) \sim \mathrm{Unif}\left\{-1,-\omega,-\omega^{2}\right\}$. One can now define a random walk on $\mathbb{H}$, 
with the increments $\left\{Y_{j}(i)\colon i=1,2, \text{ }j \geq 1\right\}$. Needless to say, this random walk has period $2$. 

\begin{cor}
\label{Thm: Hexagonal Lattice}
Consider an infinite color urn model with colors indexed by $S =\Hbold$, and kernel $R$ as given above. Suppose the
starting configuration is $U_0$. If we define, 
\[
P_n^{\text{cs}}\left(A\right) :=\frac{U_n}{n+1}\left(2 \sqrt{\log n} A \right), 
\,\,\, A \in \BB_{{\mathbb R}^d}, 
\]
then, as $n \to \infty $,
\begin{equation}
P_{n}^{cs} \stackrel{p}{\longrightarrow} \Phi_{2} \mbox{\ in\ } \PP\left(\Rbold^2\right).
\label{Equ:Asymptotic-Random-Urn-H-RW}
\end{equation}
In particular, 
\begin{equation}
\frac{2 Z_{n}} {\sqrt{\log n}} \Rightarrow \mbox{Normal}_{2}(0,\mathbb{I}_2),
\label{Equ:Asymptotic-Expected-Urn-H-RW}
\end{equation}
as $n \rightarrow \infty$
\end{cor}

\begin{proof}
First of all we note that, it is enough to show that $\overline{\varSigma} = \frac{1}{2} \mathbb{I}_2$. 

Now, since $1+\omega +\omega ^2=0$, so for the random walk on the hexagonal lattice, $\bmu(1)=\bmu(2)=0$. Therefore $\overline{\bmu}=0$. 
Let  
\begin{eqnarray*}\label{Eq: varmatrixforhexawalk} \varSigma(1):=
\begin{pmatrix}
 \sigma_{1,1} & \sigma_{1,2}\\
 \sigma_{2,1} & \sigma_{2,2}
\end{pmatrix}
\end{eqnarray*} Writing
$Y_{1}(1):=\left(Y_{1}^{(1)}(1), Y_{1}^{(2)}(1)\right)$, observe that 
\begin{equation*}
\sigma_{1,1}=\bE\left[\left(Y_{1}^{(1)}(1)\right)^2\right] \text{ and } \sigma_{2,2}=\bE\left[\left(Y_{1}^{(2)}(1)\right)^2\right]. 
\end{equation*} Also, 
\begin{equation*}
\sigma_{1,2} = \sigma_{2,1}=\bE\left[Y_{1}^{(1)}(1)Y_{1}^{(2)}(1)\right].
\end{equation*}
Writing $\omega= \mathit{Re}\left(\omega \right)+ i \mathit{Im}\left(\omega\right)$, it is easy to see that 
\begin{equation*}
\sigma_{1,1}=\frac{1}{3}\left(1+\left(\mathit{Re}\left(\omega \right)\right)^2+ \left(\mathit{Re}\left(\omega^2 \right)\right)^2\right). 
\end{equation*} Since $\mathit{Re}\left(\omega \right)= \mathit{Re}\left(\omega^2 \right)$, therefore, 
\begin{equation*}
\sigma_{1,1}=\frac{1}{3}\left(1+2\left(\mathit{Re}\left(\omega \right)\right)^2\right).
\end{equation*}
Since  
$\omega=\frac{1}{2}+\mathit{i}\frac{\sqrt{3}}{2}$, therefore, this implies 
$\sigma_{1,1}=\frac{1}{2}$. Similarly, since $\mathit{Im}\left(\omega\right)=-\mathit{Im}\left(\omega^{2}\right)$, 
\begin{equation*}
\sigma_{2,2}=\frac{1}{3}\left(\left(\mathit{Im}\left(\omega\right)\right)^2+ 
\left(\mathit{Im}\left(\omega^{2}\right)\right)^2\right)= \frac{2}{3}\left(\mathit{Im}\left(\omega\right)\right)^2=\frac{1}{2}.
\end{equation*}
Since, $\mathit{Re}\left(\omega \right)= \mathit{Re}\left(\omega^2 \right)$, and 
$\mathit{Im}\left(\omega\right)=-\mathit{Im}\left(\omega^{2}\right)$,
\begin{equation*}
\sigma_{1,2} = \sigma_{2,1}=\frac{1}{3}\left(\mathit{Re}\left(\omega \right)\mathit{Im}\left(\omega\right)+
\mathit{Re}\left(\omega^2 \right)\mathit{Im}\left(\omega^{2}\right)\right)=0.
\end{equation*}
This proves that $\varSigma(1)=\frac{1}{2}\mathbb{I}_{2}$. Similar calculations show that 
$\varSigma(2)=\frac{1}{2}\mathbb{I}_{2}$. This implies that 
$\overline{\varSigma}=\frac{1}{2}\varSigma(1)+\frac{1}{2}\varSigma(2)=\frac{1}{2}\mathbb{I}_{2}$.
This completes the proof.
\end{proof}

\section{Conclusion}
\label{Sec:Conclusion}
We have presented in this paper a new method for studying finite or infinite color \emph{balanced} urn schemes through 
their representation in the associated Markov chain. It turns out that for any general balanced urn scheme, the sequence of
observed colors is a realization of a branching Markov chain with a modified kernel on the random recursive tree. 
We have shown using such representation that under fairly general conditions one can derive asymptotic of various urn schemes, 
which otherwise may be very difficult to find. As illustrated by several examples, essentially the asymptotic may be
derived if the underlying Markov chain has a proper scaling limit. We believe that this novel approach will 
provide a better understanding of these new type of urn schemes. \\

\noindent
{\bf Remark:} While preparing the manuscript we were informed by C\'{e}cile Mailler and Jean-Fran\c{c}ois Marckert 
that they are also working on similar problems. 
We are happy to note this development and 
we hope that more attention will be given in studying infinite color urn schemes.

\section*{Acknowledgement}
The authors are grateful to Krishanu Maulik and Codina Cotar for various discussions they had with them at various time points.
We are also thankful to Tatyana Turova, whose comments helped us to fix a glitch, which we had in an earlier version of the paper.

\bibliographystyle{plain}

\bibliography{RT}

\end{document}